\documentclass{gen-j-l}

\usepackage[alphabetic]{amsrefs}
\usepackage{amssymb}
\usepackage{mathrsfs}

\newcommand{\R}{\mathbf{R}}
\newcommand{\C}{\mathbf{C}}

\newcommand{\N}{\mathbf{N}}
\renewcommand{\P}{\mathbf{P}}

\newcommand{\dist}{\operatorname{dist}}
\newcommand{\length}{\operatorname{len}}
\newcommand{\Lip}{\operatorname{Lip}}

\newcommand{\hull}{\operatorname{h}}
\renewcommand{\Re}{\operatorname{Re}}

\newcommand{\scrO}{\mathscr{O}}
\newcommand{\scrL}{\mathscr{L}}
\newcommand{\scrB}{\mathscr{B}}

\newcommand{\angles}[1]{\left< #1 \right>}
\newcommand{\pars}[1]{\left( #1 \right)}

\newcommand{\abs}[1]{\left| #1 \right|}

\newtheorem{theorem}{Theorem}[section]
\newtheorem{lemma}[theorem]{Lemma}
\newtheorem{cor}[theorem]{Corollary}

\theoremstyle{definition}

\theoremstyle{remark}

\numberwithin{equation}{section}

\begin{document}

\title{computation of some transcendental integrals from path signatures}


\author{Andrew Ursitti}
\address{Department of Mathematics, Purdue University}
\curraddr{}
\email{aursitti@math.purdue.edu}
\thanks{}

\subjclass[2010]{Primary }

\date{\today}


\begin{abstract}
It is shown that if $\gamma$ is a path of finite $p$ variation ($1\leq p< 2$) 
in a euclidean vector space and $f,g,h$ are Lipschitz functions on the trace 
of $\gamma$ then $s\mapsto F(s)=\int_\gamma f^sg dh$ defines an 
entire holomorphic function provided the convex hull of the image of $f$ does not contain 
zero. If in addition $|\log z|\leq \log 2$ on the 
convex hull of the image of $f$ then for any $s\in \C$, $F(s)$ can be computed 
from the nonnegative integer values $\{F(k)\}_{k\in \N}$. If in addition to these 
hypotheses each of $f,g,h$ is a polynomial, then the values $F(k)$ are computable 
directly from the signature of $\gamma$ thus all values of $F(s)$ are
computable from the signature. As a special case the 
winding number of a closed path $\gamma$ around an affine submanifold of codimension 
two is computed from finitely many terms of the signature provided certain estimates 
are satisfied. 
\end{abstract}

\maketitle


\section{Introduction}
In this note we will show how certain transcendental integrals of the form 
$\int_\gamma f^sg dh$ can be algorithmically recovered from the signature of the path
$\gamma$. As a 
special case we will give an alternate proof of a result due
originally to P. Yam \cite{YAM} concerning the recovery of the 
winding number of a path around a codimension two affine submanifold
from the signature of the path provided certain estimates are satisfied. 
The signature of $\gamma$ is the 
infinite tensor (i.e. formal series) 
$X_\gamma\in \widehat{\bigoplus}_{k\geq 0}V^{\otimes k}$ defined in
degree zero to be $1$ and in degree $k>0$ by the iterated integral 
of tensors 
\[
\int_{0<t_1<\cdots<t_k<T}d\gamma_{t_1}\otimes \cdots \otimes d\gamma_{t_k}.
\] 
The signature is a homomorphism from the collection of all paths 
beginning at zero of finite $p$ variation ($1\leq p< 2$), viewed 
as a group under concatenation, into the group of infinite tensors 
with $1$ in degree zero. An orientation preserving 
change of parameter does not change the signature, and an orientation 
reversing change of parameter inverts the signature, 
as was proved by K.T. Chen \cite{MR0073174} for piecewise $\mathscr{C}^1$ 
paths (the corresponding results are easily proved for paths of finite $p$ 
variation with $1\leq p <2$
using the 
Young-L\'{o}eve integration theory \cites{MR1555421,MR2604669}). Choosing 
a specific path and contatenating it with its inverse therefore produces 
a path with trivial signature, and Chen later proved \cite{MR0106258} that
concatenations of
such paths are essentially the only paths with trivial signature. 
Specifically, Chen defined a path to be \emph{irreducible} 
if it doesn't contain any segments which consist of a path and its inverse 
concatenated in succession, and then proved that 
two irreducible paths have the same signature if and only if they 
differ by a translation and an orientation 
preserving change of parameter. This was later generalized to 
paths of bounded variation by Hambly and Lyons 
\cite{MR2630037}, who defined the notion of \emph{tree-like} paths 
and proved that two paths of bounded 
variation have the same signature if and only if the concatenation 
of one with the inverse of the other 
is a lipschitz tree-like path. Boedihardjo, Ni and Qian 
\cite{MR3237773} proved that two simple paths 
of finite $p$ variation ($1\leq p< 2$) in the plane have the same signature 
if and only if they differ by translation and orientation 
preserving change of parameter. These uniqueness results show that 
one should expect topological data such as the winding number to 
be contained in the signature. 

It should be mentioned that in \cite{MR3237773} the authors also 
proved that for a closed path in the plane with variation 
less than two, the moments of the winding number when viewed 
as a function on the plane can be recovered by 
evaluating the signature on Lyndon words 
(up to some simple constant multiples). Thus, \cite{MR3237773}
essentially contains a proof that the winding number about any 
specific point can be recovered from 
the signature by first evaluating the signature on Lyndon 
words to find the moments, then computing the winding number 
at a specific point by either computing its convolution with 
a gaussian approximate identity from 
the moments and taking 
a limit, or computing 
the fourier transform from the moments, then inverting the 
fourier transform to find the winding number at 
a specific point. 

Here the winding number will be recovered from the signature 
by a different method which we
now describe. The standing hypotheses are these:
\begin{enumerate}
\item $V$ is a finite dimensional euclidean vector space over $\R$.
\item $\gamma:[0,T]\rightarrow V$ is a continuous path with finite $p$ variation, $1\leq p<2$.
\item $f:\gamma([0,T])\rightarrow \C$ is a lipschitz function from $\gamma([0,T])\subset V$ (as a metric space with metric inherited from $V$) such that the convex hull of the image $f(\gamma([0,T]))$ does 
not contain zero.
\item $g,h:\gamma([0,T])\rightarrow \C$ are arbitrary Lipschitz functions.
\item $E\in \scrO(\C)$ is an arbitrary entire function. 
\item $x_1,x_2\in V^\ast$ are independent, $\xi_1,\xi_2\in \R$ and
\[
1/2<(x_1\circ\gamma-\xi_1)^2+(x_2\circ\gamma-\xi_2)^2<2
\]
on $[0,T]$.
\end{enumerate} 
To condense notation denote by $S_\gamma$ the image 
$S_\gamma=\gamma([0,T])\subset V$ and let $\log(\cdot)$ 
be a continuous branch of the logarithm on $f(S_\gamma)$. 
Under these hypotheses we shall prove in section 2 that $(E,g,h)\mapsto I^f_\gamma(E,g,h)=\int_\gamma (E\circ \log \circ f)gdh$
defines a $\C$-trilinear map 
$\scrO(\C)\times \Lip(S_\gamma)\times \Lip(S_\gamma)\to\C$ and give 
an explicit bound on its absolute value (Corollary \ref{asdfjkl;6}). 
Choosing $E=E_{k,s}\in \scrO(\C)$ given by 
$E_{k,s}(z)=z^ke^{sz}$ defines the integral 
$I^f_\gamma(E_{k,s},g,h)=\int_\gamma (\log f)^k f^sgdh$ 
and naturally one expects that 
varying the parameter $s\in \C$ should produce an entire 
function with derivative $I^f_\gamma(E_{k+1,s},g,h)$,
this is indeed the case and is also proved in section 
2 (Theorem \ref{asdfjkl;11}). 

The entire function $s\mapsto F(s)=I^f_\gamma(E_{0,s},g,h)=\int_\gamma f^sgdh$ 
interpolates the values $\{F(k)\}_{k\in \N}$.\footnote{Here $\N=\{0,1,2,3,\ldots \}$, i.e. $0\in\N$.} In section 3 we 
use a general procedure developed by 
Boas and Buck \cites{MR0162914,MR0022601,MR0029985} 
to recover any value $F(s)$ from 
the nonnegative integer values $\{F(k)\}_{k\in\N}$, 
provided certain estimates are satisfied. Specifically, 
we shall prove:
\begin{theorem}\label{asdfjkl;1}
If 
$|\log z|<\log 2$ for all $z\in \hull(f(S_\gamma))$, then
the series
\[
\sum_{0\leq n}(-1)^n\binom{s}{n}\sum_{0\leq k\leq n}(-1)^k\binom{n}{k}\int_\gamma f^kg dh
\]
converges to $\int_\gamma f^s g dh$. 
\end{theorem}
Here $\hull(\cdot)$ denotes the convex hull. It should be noted that this 
is an iterated sum, and the order of summation should not be changed (however, 
the sum in $n$ is absolutely convergent once the inner sums in the paramter $k$ 
are computed). The inequality $|\log z|<\log 2$ is satisfied by at most one branch 
of the logarithm on $\hull(f(S_\gamma))$, since $\log 2$ is less than $2\pi$, and this is a 
crucial observation since superficially Theorem \ref{asdfjkl;1} implies that
$F(s)$, which depends on the chosen branch of the logarithm, can be computed 
from the integer-exponent values $\{F(k)\}_{k\in \N}$, which do not depend on this 
choice. 

If $f,g$ and $h$
are polynomials then $F(k)=\int_\gamma f^kgdh$ can be 
extracted directly from the signature of $\gamma$. 
Specifically, if $\gamma(0)=0$ and 
if $x_1,\ldots, x_d$ is a basis for $V^\ast$ 
and $\alpha,\beta\in \N^d$ are multi-indices, 
then 
\begin{equation}\label{asdfjkl;200}
\int_\gamma x^\alpha d(x^\beta)=\sum_{i\leq d}\beta_i\sum_{\sigma\in\mathfrak{S}_{|\alpha|+|\beta|-1}}\angles{\sigma[(x_1^{\otimes (\alpha_1+\beta_1)}\otimes\cdots \otimes x_d^{\otimes (\alpha_d+\beta_d)})/x_i]\otimes x_i,X_\gamma}.
\end{equation}

In this expression, 
$(x_1^{\otimes (\alpha_1+\beta_1)}\otimes\cdots \otimes x_d^{\otimes (\alpha_d+\beta_d)})/x_i$ 
indicates the tensor 
$x_1^{\otimes (\alpha_1+\beta_1)}\otimes\cdots \otimes x_d^{\otimes (\alpha_d+\beta_d)}$ with exactly \emph{one} factor $x_i$ removed. 
If $\beta_i>0$ (which is the only case that matters) 
at least one factor of $x_i$ appears inside of a 
consecutive list of such factors in $x_1^{\otimes (\alpha_1+\beta_1)}\otimes\cdots \otimes x_d^{\otimes (\alpha_d+\beta_d)}$, so this ``division" 
operation is well defined. The entire argument 
$\sigma[(x_1^{\otimes (\alpha_1+\beta_1)}\otimes\cdots \otimes x_d^{\otimes (\alpha_d+\beta_d)})/x_i]\otimes x_i$
can be interpreted as follows: from 
$x_1^{\otimes (\alpha_1+\beta_1)}\cdots \otimes x_d^{\otimes (\alpha_d+\beta_d)}$, remove one factor $x_i$, let the permutation $\sigma$ permute the 
remaining $|\alpha|+|\beta|-1$ factors, then 
replace the factor $x_i$ on the right. Integrals of the form $\int_\gamma PdQ$ where $P$ and $Q$ are polynomials can 
then be computed by splitting $P$ and $Q$ into monomials and using (\ref{asdfjkl;200}). With minor adjustments one can do 
away with the requirement $\gamma(0)=0$.

In particular, if $f,g$ and $h$ are polynomials then the values $\{F(k)\}_{k\in\N}$ can be extracted 
directly from $X_\gamma$ so evidently Theorem \ref{asdfjkl;1} shows that if $|\log z|<\log 2$ on the 
convex hull of the trace of $f\circ \gamma$, then every value of the entire function $F(s)=\int_\gamma f^sgdh$
can be recovered from the signature of $\gamma$. In particular, if in addition we assume that 
$\gamma$ is a closed path then we can use this method to recover the 
winding number of $\gamma$ around the codimension two affine submanifold $\{x_1=\xi_1,x_2=\xi_2\}$ provided
the standing hypothesis $1/2<(x_1\circ\gamma-\xi_1)^2+(x_2\circ\gamma-\xi_2)^2<2$ is satisfied. The $x_1\wedge x_2$-oriented winding number of 
$\gamma$ around 
$\{x_1=\xi_1,x_2=\xi_2\}$ is given by 
\[
W_\gamma(x_1\wedge x_2;\xi_1,\xi_2)=\frac{1}{2\pi}\int_\gamma\frac{(x_1-\xi_1) dx_2-(x_2-\xi_2) dx_1}{(x_1-\xi_1)^2+(x_2-\xi_2)^2}
\]
Thus, further specifying the parameters to 
$f=(x_1-\xi_1)^2+(x_2-\xi_2)^2$, $g=x_1-\xi_1$, $h=x_2-\xi_2$, and then switching 
$g$ and $h$ for the second summand, we find that 
\[
F(s)=\int_\gamma [(x_1-\xi_1)^2+(x_2-\xi_2)^2]^s ((x_1-\xi_1) dx_2-(x_2-\xi_2) dx_1)
\]
can be recovered from the signature 
by Theorem \ref{asdfjkl;1}, provided that 
$1/2<(x_1\circ\gamma-\xi_1)^2+(x_2\circ\gamma-\xi_2)^2<2$. In particular the winding number
$\frac{1}{2\pi}F(-1)= W_\gamma(x_1\wedge x_2;\xi_1,\xi_2)$ 
can be found in this manner. Specifically, we prove:
\begin{theorem}\label{asdfjkl;101}
If in addition to the standing hypotheses $\gamma$ is a closed path,
then 
\begin{align*}
W_\gamma&(x_1\wedge x_2;\xi_1,\xi_2) \\
&=\frac{1}{2\pi}\sum_{0\leq n\leq N}\sum_{0\leq k\leq n}(-1)^{k}\binom{n}{k}\sum_{k_1+k_2=k}\frac{k!}{k_1!k_2!} \sum_{\begin{subarray}{c} j_1^1+j_1^2=2k_1 \\ j_2^1+j_2^2=2k_2\end{subarray}} 
\frac{(2k_1)!}{j_1^1!j_1^2!}\frac{(2k_2)!}{j_2^1!j_2^2!} \angles{T^{j_1^1,j_1^2}_{j_2^1,j_2^2},X_\gamma}
\end{align*}
where $T^{j_1^1,j_1^2}_{j_2^1,j_2^2}\in V^{\ast\otimes(j_1^1+j_2^2+1)}\bigoplus V^{\ast\otimes(j_1^1+j_2^2+2)}$ is defined in (\ref{asdfjkl;300}).
\end{theorem}

In addition to this it is shown that 
$W_\gamma(x_1\wedge x_2;\xi_1,\xi_2)$ can be computed from only 
finitely many terms in the signature, and an estimate on how many
terms are necessary is given. All of this is detailed in section 4.

\section{Regularity of $I^{f}_\gamma$}
To condense notation, define
\begin{align*}
M(E,f,g)&= 
\|g\|_{\Lip(S_\gamma)}\max_{f(S_\gamma)}|E\circ \log |\\
&\hspace{1cm}
+\frac{\|f\|_{\Lip(S_\gamma)}\max_{S_\gamma}|g|\max_{\hull(f(S_\gamma))}|E'\circ \log|}{\dist(0,\hull(f(S_\gamma)))}
\end{align*}
Regarding $\dist(0,\hull(f(S_\gamma)))$, we note that $f(S_\gamma)$
is compact so $\hull(f(S_\gamma))$ is closed and since zero is not in $\hull(f(S_\gamma))$ by
hypothesis, $\dist(0,\hull(f(S_\gamma)))$ is positive.
\begin{lemma}\label{asdfjkl;7}
The map $(E\circ \log \circ f)g:S_\gamma\rightarrow \C$ is Lipschitz on $S_\gamma$ and satisfies $\|(E\circ \log \circ f)g\|_{\Lip(S_\gamma)}\leq M(E,f,g)$.
\end{lemma}

\begin{proof}
Let $w_1,w_2\in \C$, $z_1,z_2\in \hull(f(S_\gamma))$ and let $l_{z_1}^{z_2}\subset  \hull(f(S_\gamma))$ denote the oriented line segment connecting $z_1$ to $z_2$ so that  
\begin{align*}
|E(\log z_2)w_2&-E(\log z_1)w_1|\\
&=|w_2(E(\log z_2)-E(\log z_1))+E(\log z_1)(w_2-w_1)|\\
&\leq |w_2||E(\log z_2)-E(\log z_1)|+|E(\log z_1)||w_2-w_1|\\
&= |w_2|\left|\int_{l_{z_1}^{z_2}}E'(\log z)z^{-1}dz\right|+|E(\log z_1)||w_2-w_1|\\
&\leq |w_2|\pars{\max_{z\in l_{z_1}^{z_2}}|E'(\log z)z^{-1}|}|z_2-z_1|+|E(\log z_1)||w_2-w_1|
\end{align*}
For any $t_1,t_2\in [0,T]$, we can use this estimate with $w_1=g(\gamma(t_1))$, $w_2=g(\gamma(t_2))$, $z_1=f(\gamma(t_1))$ and $z_2=f(\gamma(t_2))$
 to write 
\begin{align*}
|E&(\log f(\gamma(t_2)))g(\gamma(t_2))-E(\log f(\gamma(t_1)))g(\gamma(t_1))|\\
&\leq |g(\gamma(t_2))|\pars{\max_{z\in l_{f(\gamma(t_1))}^{f(\gamma(t_2))}}|E'(\log z)z^{-1}|}|f(\gamma(t_2))-f(\gamma(t_1))|\\
&\hspace{2cm}+|E(\log f(\gamma(t_1)))||g(\gamma(t_2))-g(\gamma(t_1))|\\
&\leq \max_{S_\gamma}|g|\max_{\hull(f(S_\gamma))}|(E'\circ \log (\cdot))(\cdot)^{-1}|\|f\|_{\Lip(S_\gamma)}|\gamma(t_2)-\gamma(t_1)|\\
&\hspace{2cm}+\max_{f(S_\gamma)}|E\circ \log |\|g\|_{\Lip(S_\gamma)}|\gamma(t_2)-\gamma(t_1)|\\
&\leq \max_{S_\gamma}|g|\max_{\hull(f(S_\gamma))}|E'\circ \log |\|f\|_{\Lip(S_\gamma)} \\
&\hspace{1cm}\times\dist(0,\hull(f(S_\gamma)))^{-1}|\gamma(t_2)-\gamma(t_1)|\\
&\hspace{2cm}+\max_{f(S_\gamma)}|E\circ \log |\|g\|_{\Lip(S_\gamma)}|\gamma(t_2)-\gamma(t_1)|\\
&=M(E,f,g)|\gamma(t_2)-\gamma(t_1)|.
\end{align*}
\end{proof}

Combining Lemma \ref{asdfjkl;7} with the Young-L\'{o}eve integration theory \cites{MR1555421,MR2604669}, we have the following corollary:
\begin{cor} \label{asdfjkl;6}The expression 
\[
 I^f_\gamma(E,g,h)=\int_\gamma (E\circ \log \circ f)gdh
\]
defines a $\C$-trilinear map $I^f_\gamma(\cdot,\cdot,\cdot):\scrO(\C)\times \Lip(S_\gamma)\times \Lip(S_\gamma)\rightarrow \C$ which satisfies the estimate
\begin{align*}
|I^f_\gamma (E,g,h)| &\leq \frac{1}{1-2^{1-2/p}} M(E,f,g)\|h\|_{\Lip(S_\gamma)}\|\gamma\|_{p;[0,T]}^2 \\
&\hspace{.5cm} +|E\circ\log\circ f\circ \gamma(0)||g\circ \gamma(0)||h\circ\gamma(T)-h\circ\gamma(0)|.
\end{align*}
\end{cor}
\begin{proof}
From Young's estimate (or rather a variation thereof presented in \cite{MR2604669}, e.g.), 
\begin{align*}
|I^f_\gamma (E,g,h)| &\leq \frac{1}{1-2^{1-2/p}} \|(E\circ \log \circ f\circ \gamma)(g\circ \gamma)\|_{p;[0,T]}\|h\circ\gamma\|_{p;[0,T]} \\
&\hspace{.5cm} +|E\circ\log\circ f\circ \gamma(0)||g\circ \gamma(0)||h\circ\gamma(T)-h\circ\gamma(0)|\\
&\leq \frac{1}{1-2^{1-2/p}} M(E,f,g)\|h\|_{\Lip(S_\gamma)}\|\gamma\|_{p;[0,T]}^2 \\
&\hspace{.5cm} +|E\circ\log\circ f\circ \gamma(0)||g\circ \gamma(0)||h\circ\gamma(T)-h\circ\gamma(0)|.
\end{align*}
\end{proof}

Now we would like to consider simultaneously the family of entire functions $\{E_{k,s}:k\in \N,s\in\C\}$ given by $E_{k,s}(z)=z^ke^{sz}$, thus producing the integrals
\[
I^f_\gamma(E_{k,s},g,h)=\int_\gamma (E_{k,s}\circ \log \circ f)gdh=
\int_\gamma (\log  f)^ke^{s\log f}gdh=
\int_\gamma (\log  f)^kf^s gdh.
\]

\begin{theorem}\label{asdfjkl;11}
For any $k\in\N$, $s\mapsto I^f_\gamma(E_{k,s},g,h)$ defines an entire function.
\end{theorem}
\begin{proof}
First, it must be proved that $I^f_\gamma(E_{k,(\cdot)},g,h)$ is differentiable. The natural guess for the derivative is of course $I^f_\gamma(E_{k+1,(\cdot)},g,h)$ so we attempt to verify the asymptotic equality 
\[
I^f_\gamma(E_{k,s},g,h)=I^f_\gamma(E_{k,s_o},g,h)+(s-s_o)I^f_\gamma(E_{k+1,s_o},g,h)
+o(|s-s_o|)
\]
for every $s_o\in\C$. However, since $E\mapsto I^f_\gamma(E,g,h)$ is $\C$-linear this is implied by $|I^f_\gamma(\widetilde{E}_{k,s,s_o},g,h)|=o(|s-s_o|)$
for every $s_o\in\C$ where $\widetilde{E}_{k,s,s_o}\in\scrO(\C)$ is given by 
\[
\widetilde{E}_{k,s,s_o}(z)=z^ke^{sz}-z^ke^{s_oz}-(s-s_o)z^{k+1}e^{s_oz}=z^ke^{s_oz}(e^{(s-s_o)z}-1-(s-s_o)z).
\]
The rightmost expression shows that $|\widetilde{E}_{k,s,s_o}(z)|=O(|s-s_o|^2)$ pointwise for every $z$ and uniformly for $z$ in any bounded subset of $\C$. In particular 
\begin{equation}\label{asdfjkl;10}
|\widetilde{E}_{k,s,s_o}\circ \log \circ f(\gamma(0))|=o(|s-s_o|)
\end{equation}
and since $\log$ must map the compact set $f(S_\gamma)$ into another compact set, 
\begin{equation}\label{asdfjkl;8}
\max_{z\in f(S_\gamma)}|\widetilde{E}_{k,s,s_o}(\log z)|=o(|s-s_o|).
\end{equation}
Also, 
\begin{align*}
\widetilde{E}'_{k,s,s_o}(z)&=kz^{k-1}e^{s_oz}(e^{(s-s_o)z}-1-(s-s_o)z)\\
&\hspace{.5cm}+z^ks_oe^{s_oz}(e^{(s-s_o)z}-1-(s-s_o)z)
+z^ke^{s_oz}(s-s_o)(e^{(s-s_o)z}-1)
\end{align*}
thus $|\widetilde{E}'_{k,s,s_o}(z)|=O(|s-s_o|^2)$ pointwise and uniformly in bounded subsets. In particular
\begin{equation}\label{asdfjkl;9}
\max_{z\in \hull(f(S_\gamma))}|\widetilde{E}'_{k,s,s_o}(\log z)|=o(|s-s_o|).
\end{equation}
On combining  (\ref{asdfjkl;10}), (\ref{asdfjkl;8}) and (\ref{asdfjkl;9}), $|I^f_\gamma(\widetilde{E}_{k,s,s_o}, g,h)|=o(|s-s_o|)$ by the estimate given in Corollary \ref{asdfjkl;6}. This proves that $I^f_\gamma(E_{k,(\cdot)},g,h)\in \scrO(\C)$ with derivative $I^f_\gamma(E_{k+1,(\cdot)},g,h)$.
\end{proof}

\section{Proof of theorem \ref{asdfjkl;1}}
Recall from the introduction that $\gamma:[0,T]\to V$ is a path of $p$ variation $(1\leq p<2)$ taking values in the euclidean vector space $V$, with signature $X_\gamma\in \widehat{\bigoplus}_{k\geq 0} V^{\otimes k}$, as described in the introduction. The functions $f,g,h:S_\gamma\to \C$ are 
Lipschitz and the convex hull of the image of $f$ does not contain zero. Our task in this 
section is to prove Theorem \ref{asdfjkl;1}, so we will herein assume that $\log:\hull(f(S_\gamma))\to \C$, if it exists, is the unique branch of the logarithm such that $|\log z|<\log 2$ on $\hull(f(S_\gamma))$. Such a unique logarithm exists, for instance, if $f$ is positive 
and satisfies $1/2<f<2$ on $S_\gamma$, for then $\hull(f(S_\gamma))=f(S_\gamma)$ is a compact  subinterval of $(1/2,2)$ whence $|\log f|<\log 2$. 

Thus, with $F(s)=I^f_\gamma(E_{0,s},g,h)=\int_\gamma f^sgdh$ as in the introduction, we are tasked with 
computing the value $F(s)$ from the known values $\{F(k)\}_{k\in\N}$ which 
come directly from the signature. There is a general procedure developed by Boas and Buck \cites{MR0162914,MR0022601,MR0029985} which can accomplish this task, provided that $F$ satisfies certain growth conditions at infinity. It seems appropriate to briefly describe the procedure rather than simply quoting the relevant results. If $H\in\scrO(\C)$ is a generic entire function which satisfies an estimate of the form $|H(z)|\leq Ae^{B|z|}$ for $z\in (1,\infty)$ then its Laplace transform 
$\scrL H(w)=\int_0^\infty H(z)e^{-wz}dz$
defines a holomorphic function in the region $\{\Re w>B\}$, and it is natural to ask if $\scrL H$ extends to a holomorphic function in the neighborhood of infinity defined by $\{|w|>B\}$. If this is the case then it is easy to deduce what the Taylor coefficients at infinity must be since if $w\in (B,\infty)$ then
\begin{align*}
\scrL H(w)&=\int_0^\infty H(z)e^{-wz}dz\\
&=\sum_{n\geq 0}\frac{H^{(n)}(0)}{n!}\int_0^m z^ne^{-wz}dz+\int_m^\infty H(z)e^{-wz}dz\\
&=\sum_{n\geq 0}\frac{H^{(n)}(0)}{n!}\frac{1}{w^{n+1}}\int_0^m (wz)^{(n+1)-1}e^{-wz}d(wz)+\int_m^\infty H(z)e^{-wz}dz.
\end{align*}
By letting $m\rightarrow \infty$ the remainder tends to zero and we recognize $\Gamma(n+1)=n!$ in each term so that $\scrL H(w)=\sum_{n\geq 0}H^{(n)}(0)/w^{n+1}$ on $(B,\infty)$. Therefore, $\scrL H $ will extend to the region $\{|w|>B\}$ provided that $\limsup_{n\to\infty} |H^{(n)}(0)|^{1/n}< B$ 
for then
\[
\limsup_{n\to \infty}\abs{H^{(n)}(0)/w^{n+1}}^{1/n}
=
\frac{1}{|w|}\limsup_{n\to \infty} \frac{|H^{(n)}(0)|^{1/n}}{|w|^{1/n}}
<1
\]
This will be the case if the estimate $|H(z)|< Ae^{B|z|}$ holds for all $z$ and not only for $z\in (1,\infty)$, for then by Cauchy's estimate $|H^{(n)}(0)|< n!Ae^{Br}/r^n$ for all $r>0$ and this is minimized at $r=n/B$ so that $|H^{(n)}(0)|< n!AB^ne^{n}/n^n$. Therefore
\begin{align*}
|H^{(n)}(0)|^{1/n}&< A^{1/n}B\frac{e(n!)^{1/n}}{n}=(2\pi n)^{1/2n}A^{1/n}B\frac{e(n!)^{1/n}}{n(2\pi n)^{1/2n}}
\end{align*}
so that $\limsup_{n\to\infty}|H^{(n)}(0)|^{1/n}< B$, by Stirling's estimate. Thus, the power series $\scrB H(w)=\sum_{n\geq 0}H^{(n)}(0)/w^{n+1}$ converges absolutely to an analytic function, uniformly on compact subsets of the region $\{|w|>B\}$, and therefore defines a holomorphic function in a neighborhood of $\infty \in \P^1_\C$, taking the value $0$ at $\infty$ and extending $\scrL H$. The extension $\scrB H$ of $\scrL H$ is usually referred to as the Borel transform of $H$.

We can invert this procedure as follows. If $r>B$ then for fixed $z$ both power series $\scrB H(w)=\sum_{n\geq 0}H^{(n)}(0)/w^{n+1}$ and $e^{zw}=\sum_{n\geq 0}z^nw^n/n!$ converge absolutely and uniformly on the circle $\{|w|=r\}$ and therefore
\[
\int_{|w|=r}\scrB H(w)e^{zw}dw=\sum_{n\geq 0}\pars{\sum_{k\geq0}\frac{H^{(k)}(0)}{n!}\int_{|w|=r}w^{n-k-1}dw}z^n,
\]
but $\int_{|w|=r}w^{n-k-1}dw$ is nonzero only if $n=k$ so evidently 
\begin{equation}
H(z)=\frac{1}{2\pi i}\int_{|w|=r}\scrB H(w)e^{zw}dw.
\end{equation} 
This is called the P\'{o}lya representation of $H$, it is valid not only for the contour $\{|w|=r\}$ but any simple closed contour contained in $\{|w|>B\}$ and it suggests a generalization, due to Buck \cites{MR0022601,MR0029985}, which will allow us to compute any value $H(s)$ from $\{H(k)\}_{k\in \N}$ and thus prove Theorem \ref{asdfjkl;1} by substituting $F(s)=\int_\gamma f^sgdh$ for $H$. The essence of Buck's method is that rather than settling only for the series expansion $e^{zw}=\sum_{n\geq 0}z^nw^n/n!$, we can choose to write $e^{zw}$ in any of a number of different ways. In particular we will be interested in the binomial series expansion:
\[
e^{zw}=(e^w)^z=(e^w-1+1)^z
=\sum_{n\geq 0}\binom{z}{n}(e^{w}-1)^n,
\]
which is valid in the region defined by $|e^w-1|<1$. 

\begin{lemma}\label{asdfjkl;20}
The inclusion $\{|w|=r\}\subset \{|e^w-1|<1\}$ holds if and only if $r<\log 2$.
\end{lemma}
\begin{proof}
If $w=x+iy$ then $|e^w-1|^2=e^{2x}-2e^x\cos y+1$, so the first observation to be made is that
if $|w|=r$ and $|e^w-1|<1$ then $r<\pi/2$ necessarily, for otherwise the circle 
$\{|x+iy|=r\}$ contains points with $\cos y<0$ which would imply $|e^w-1|^2=e^{2x}-2e^x\cos y+1>1$. Having reduced consideration to $r<\pi/2$, we observe that $(x,y)\mapsto e^{2x}-2e^x\cos y+1$ can achieve a maximum at a point $(x,y)$ in the circle $x^2+y^2=r^2$ only if its gradient is orthogonal to $(y,-x)$, or in other words only if $ye^x-y\cos y-x\sin y=0$. Since $(\pm r, 0)$ can be checked individually we only care about the case $0<|y|<\pi/2$ and the necessary condition in this case reduces to $e^x-\cos y-(\sin y/y) x=0$ with $\cos y,\sin y/y>0$ so that the minimum value of $x\mapsto e^x-\cos y-(\sin y/y) x$ is $(\sin y/y)-\cos y -(\sin y/y)\log (\sin y/y)$, but this is positive for $y\in (-\pi/2,0)\cup (0,\pi/2)$ and so $ye^x-y\cos y-x\sin y=0$ and $|y|<\pi/2$ imply $y=0$. Thus, the extremal values of $(x,y)\mapsto e^{2x}-2e^x\cos y+1$ on the circle $\{|x+iy|=r\}$ 
must occur at $(\pm r, 0)$. The maximum and minimum values are therefore $e^{ 2r}-2e^{ r}+1$ and 
$e^{-2r}-2e^{-r}+1$ respectively and one 
finds $r=\log 2$ as the threshold value for the inclusion of sets stated in the lemma. 
\end{proof}

So, if $H$ is such that $|H(z)|\leq Ae^{B|z|}$ with $B<\log 2$ then $r$ can be chosen such that $\{|w|=r\}$ lies in both the region of absolute convergence of the Borel transform $\scrB H$ \emph{and} the region of absolute convergence of the series $e^{zw}=\sum_{n\geq 0}\binom{z}{n}(e^w-1)^n$ and therefore
\begin{align*}
H(z)&=\frac{1}{2\pi i}\int_{|w|=r}\scrB H(w)e^{zw}dw \\
&=\sum_{n\geq 0}\binom{z}{n}\frac{1}{2\pi i}\int_{|w|=r}\scrB H(w)(e^w-1)^ndw \\
&=\sum_{n\geq 0}\binom{z}{n}\sum_{0\leq k\leq n}\binom{n}{k}(-1)^{n-k}\frac{1}{2\pi i}\int_{|w|=r}\scrB H(w)e^{kw}dw \\
&=\sum_{n\geq 0}\binom{z}{n}\sum_{0\leq k\leq n}\binom{n}{k}(-1)^{n-k}H(k) \\
\end{align*}
and therefore 
\begin{equation}\label{asdfjkl;14}
H(z)=\sum_{0\leq n}(-1)^n\binom{z}{n}\sum_{0\leq k\leq n}(-1)^{k}\binom{n}{k}H(k).
\end{equation}

To finish the proof of Theorem \ref{asdfjkl;1} we need only to observe that 
the estimate $|\log f|<\log 2$ implies the required growth condition $|F(s)|\leq Ae^{B|s|}$ 
with $B<\log 2$. This is a simple consequence of the results of section 2, specifically Corollary \ref{asdfjkl;6}.

\section{The winding number}
Recall from the introduction that if in addition to the standing hypotheses we assume that $\gamma$ is a 
closed path, then
\[
W_\gamma(x_1\wedge x_2;\xi_1,\xi_2)=\frac{1}{2\pi}\int_\gamma\frac{(x_1-\xi_1) dx_2-(x_2-\xi_2) dx_1}{(x_1-\xi_1)^2+(x_2-\xi_2)^2}
\]
is the $x_1\wedge x_2$-oriented winding number around the codimension two affine submanifold $\{x_1=\xi_1,x_2=\xi_2\}$ and it can be computed using Theorem \ref{asdfjkl;1} and the signature $X_\gamma$ since $W_\gamma(x_1\wedge x_2;\xi_1,\xi_2)=\frac{1}{2\pi} F(-1)$ where $F$ is the entire function defined by 
\[
F(s)=\int_\gamma[(x_1-\xi_1)^2+(x_2-\xi_2)^2]^s ((x_1-\xi_1)dx_2-(x_2-\xi_2) dx_1).
\]
Our first task in this section is to prove Theorem \ref{asdfjkl;101}.
By (\ref{asdfjkl;14}),
\[
2\pi W_\gamma(x_1\wedge x_2;\xi_1,\xi_2) 
=\sum_{0\leq n}\sum_{0\leq k\leq n}(-1)^k\binom{n}{k}F(k)
\]
provided the standing hypothesis $1/2<(x_1-\xi_1)^2+(x_2-\xi_2)^2<2$ is satisfied.
Now the values $F(k)$ for $k\in\N$ can be computed from the signature in a rather explicit fashion using (\ref{asdfjkl;200}):

\begin{align}
F(k)&=\int_\gamma[(x_1-\xi_1)^2+(x_2-\xi_2)^2]^k((x_1-\xi_1)dx_2-(x_2-\xi_2) dx_1) \notag \\
&=\sum_{k_1+k_2=k}\frac{k!}{k_1!k_2!}\int_\gamma(x_1-\xi_1)^{2k_1}(x_2-\xi_2)^{2k_2}((x_1-\xi_1)dx_2-(x_2-\xi_2) dx_1)\notag \\
&=\sum_{k_1+k_2=k}\frac{k!}{k_1!k_2!} \sum_{\begin{subarray}{c} j_1^1+j_1^2=2k_1 \\ j_2^1+j_2^2=2k_2\end{subarray}} 
\frac{(2k_1)!}{j_1^1!j_1^2!}\frac{(2k_2)!}{j_2^1!j_2^2!} \angles{T^{j_1^1,j_1^2}_{j_2^1,j_2^2},X_\gamma}\notag
\end{align}
where $T^{j_1^1,j_1^2}_{j_2^1,j_2^2}\in V^{\ast\otimes(j_1^1+j_2^2+1)}\bigoplus V^{\ast\otimes(j_1^1+j_2^2+2)}$
is the dual tensor
\begin{align}
T^{j_1^1,j_1^2}_{j_2^1,j_2^2}
&=(x_1\circ\gamma(0)-\xi_1)^{j_1^2}(x_2\circ\gamma(0)-\xi_2)^{j_2^1} \label{asdfjkl;300}\\
&\hspace{.5cm}\times
 \left( \sum_{\sigma\in\mathfrak{S}_{j_1^1+j_2^2+1}}
\sigma[x_1^{\otimes(j_1^1+1)}\otimes x_2^{\otimes j_2^2}]\otimes x_2-\sigma[x_1^{\otimes j_1^1}\otimes x_2^{\otimes (j_2^2+1)}]\otimes x_1\right. \notag \\
&\hspace{2.5cm}+
  (x_1\circ\gamma(0)-\xi_1)\sum_{\sigma\in\mathfrak{S}_{j_1^1+j_2^2}}
\sigma[x_1^{\otimes j_1^1}\otimes x_2^{\otimes j_2^2}]\otimes x_2 \notag \\
&\hspace{2.5cm}-
  \left. (x_2\circ\gamma(0)-\xi_2)\sum_{\sigma\in\mathfrak{S}_{j_1^1+j_2^2}}
  \sigma[x_1^{\otimes j_1^1}\otimes x_2^{\otimes  j_2^2}]\otimes x_1\right) \notag 
\end{align}
This completes the proof of Theorem \ref{asdfjkl;101}.

For computational purposes one should exploit the fact that the winding number is an integer, and as such it is known once it is known within an error strictly less than $1/2$. Specifically, if $-\log 2 < -r < \log[(x_1-\xi_1)^2+(x_2-\xi_2)^2] < r <\log 2$
then for any $N$, 
\begin{align*}
&\abs{2\pi W_\gamma(x_1\wedge x_2;\xi_1,\xi_2)   -\sum_{0\leq n\leq N}\sum_{0\leq k\leq n}(-1)^{k}\binom{n}{k}F(k)} \\
&\hspace{3cm}=\abs{\sum_{N+1\leq n}\sum_{0\leq k\leq n}(-1)^{k}\binom{n}{k}F(k)} \\
&\hspace{3cm}=\abs{\sum_{N+1\leq n}\frac{(-1)^{n}}{2\pi i}\int_{|w|=r}\scrB F(w)(e^w-1)^ndw} \\
&\hspace{3cm}\leq \frac{1}{2\pi}\|\scrB F\|_{L^1(\{|w|=r\})}\sum_{N+1\leq n}\|e^{(\cdot)}-1\|_{L^\infty(\{|w|=r\})}^n \\
&\hspace{3cm}=\frac{1}{2\pi}\|\scrB F\|_{L^1(\{|w|=r\})}\frac{(e^{2r}-2e^{r}+1)^{N+1}}{2e^{r}-e^{2r}}
\end{align*}

by Lemma \ref{asdfjkl;20}. We have proved:

\begin{cor}\label{asdfjkl;100}
If $N\in\N$, $-\log 2<-r< \log[(x_1-\xi_1)^2+(x_2-\xi_2)^2] < r<\log 2$
and
\[
\|\scrB F\|_{L^1(\{|w|=r\})}\frac{(e^{2r}-2e^{r}+1)^{N+1}}{2e^{r}-e^{2r}}<2\pi^2
\]
then $W_\gamma(x_1\wedge x_2;\xi_1,\xi_2)$ is equal to the integer nearest the finite sum 
\[
\frac{1}{2\pi}\sum_{0\leq n\leq N}\sum_{0\leq k\leq n}(-1)^{k}\binom{n}{k}\sum_{k_1+k_2=k}\frac{k!}{k_1!k_2!} \sum_{\begin{subarray}{c} j_1^1+j_1^2=2k_1 \\ j_2^1+j_2^2=2k_2\end{subarray}} 
\frac{(2k_1)!}{j_1^1!j_1^2!}\frac{(2k_2)!}{j_2^1!j_2^2!} \angles{T^{j_1^1,j_1^2}_{j_2^1,j_2^2},X_\gamma}
\]
where $T^{j_1^1,j_1^2}_{j_2^1,j_2^2}\in V^{\ast\otimes(j_1^1+j_2^2+1)}\bigoplus V^{\ast\otimes(j_1^1+j_2^2+2)}$ is defined in (\ref{asdfjkl;300}).
\end{cor}
The winding number is therefore computable from only finitely many terms in the signature, and an estimate on the number of terms needed can be computed directly from an estimate of $\|\scrB F\|_{L^1(\{|w|=r\})}$. Such an estimate can be obtained in the general case $1\leq p<2$ from Corollary \ref{asdfjkl;6}, but we will only state the result precisely for the bounded variation case. 

If $\log[(x_1-\xi_1)^2+(x_2-\xi_2)^2] \leq \rho$ then uniformly on $\gamma([0,T])$,
\[
|x_1-\xi_1|\leq \sqrt{(x_1-\xi_1)^2+(x_2-\xi_2)^2}\leq e^{\rho/2}
\]
and likewise $|x_2-\xi_2|\leq e^{\rho/2}$. Therefore, if in addition to the standing hypotheses
we also assume that $\gamma$ is of bounded variation then
\begin{align*}
|F^{(n)}(0)|&=\abs{ \int_\gamma(\log [(x_1-\xi_1)^2+(x_2-\xi_2)^2])^n((x_1-\xi_1) dx_2-(x_2-\xi_2)dx_1)} \\
&\leq \rho^ne^{\rho/2}(\length(x_2 \circ \gamma)+\length(x_1 \circ \gamma))
\end{align*}
provided that the lower bound $-\rho\leq \log[(x_1-\xi_1)^2+(x_2-\xi_2)^2]$ holds as well (so that
$\log[(x_1-\xi_1)^2+(x_2-\xi_2)^2]$ is bounded in absolute value by $\rho$).
Thus, if $\rho<|w|$ then 
\begin{align*}
|\scrB F(w)|&=\abs{\sum_{n\geq 0}F^{(n)}(0)/w^{n+1}} \\
&\leq \sum_{n\geq 0}\frac{\rho^ne^{\rho/2}(\length(x_2 \circ \gamma)+\length(x_1 \circ \gamma))}{|w|^{n+1}} \\
&=\frac{e^{\rho/2}(\length(x_2 \circ \gamma)+\length(x_1 \circ \gamma))}{|w|-\rho} 
\end{align*}
and therefore if $-\log 2<-r<\rho \leq  \log[(x_1-\xi_1)^2+(x_2-\xi_2)^2] \leq \rho < r<\log 2$
then
\begin{align*}
&\abs{2\pi W_\gamma(x_1\wedge x_2;\xi_1,\xi_2)   -\sum_{0\leq n\leq N}\sum_{0\leq k\leq n}(-1)^{k}\binom{n}{k}F(k)} \\
&\hspace{3cm}\leq \frac{1}{2\pi} \|\scrB F\|_{L^1(\{|w|=r\})}\frac{(e^{2r}-2e^{r}+1)^{N+1}}{2e^{r}-e^{2r}}\\
&\hspace{3cm}\leq  r\frac{e^{\rho/2}(\length(x_2 \circ \gamma)+\length(x_1 \circ \gamma))}{r-\rho}\frac{(e^{2r}-2e^{r}+1)^{N+1}}{2e^{r}-e^{2r}}.
\end{align*}

We have proved:
\begin{cor}\label{asdfjkl;30}
If in addition to the standing hypotheses, we also assume that $\gamma$ is of bounded variation, and if $N\in\N$ and $\rho,r>0$ are chosen so that 
\[
-\log 2<-r <-\rho\leq \log[(x_1-\xi_1)^2+(x_2-\xi_2)^2] \leq \rho < r<\log 2
\]
and  
\[
r\frac{e^{\rho/2}(\length(x_2 \circ \gamma)+\length(x_1 \circ \gamma))}{r-\rho}\frac{(e^{2r}-2e^{r}+1)^{N+1}}{2e^{r}-e^{2r}}<\pi
\]
then
$W_\gamma(x_1\wedge x_2;\xi_1,\xi_2)$ is equal to the integer nearest the finite sum 
\[
\frac{1}{2\pi}\sum_{0\leq n\leq N}\sum_{0\leq k\leq n}(-1)^{k}\binom{n}{k}\sum_{k_1+k_2=k}\frac{k!}{k_1!k_2!} \sum_{\begin{subarray}{c} j_1^1+j_1^2=2k_1 \\ j_2^1+j_2^2=2k_2\end{subarray}} 
\frac{(2k_1)!}{j_1^1!j_1^2!}\frac{(2k_2)!}{j_2^1!j_2^2!} \angles{T^{j_1^1,j_1^2}_{j_2^1,j_2^2},X_\gamma}
\]
where $T^{j_1^1,j_1^2}_{j_2^1,j_2^2}\in V^{\ast\otimes(j_1^1+j_2^2+1)}\bigoplus V^{\ast\otimes(j_1^1+j_2^2+2)}$ is defined in (\ref{asdfjkl;300}).
\end{cor}


\end{document}